\documentclass{amsart}
\usepackage{amsmath}
\usepackage{amssymb}
\usepackage{graphicx}
\input xy
\xyoption{all}
\newtheorem{theorem}{Theorem}

\newtheorem{lemma}[theorem]{Lemma}
\newtheorem{corollary}[theorem]{Corollary}
\newtheorem{proposition}[theorem]{Proposition}
\theoremstyle{definition}
\newtheorem{remark}{Remark}

\newcommand{\Char}{\mbox{\rm char}}

 \newcommand{\sign}{\operatorname{sign}}

\newcommand{\Z}{\mathbb{Z}}

\begin{document}
\title[Generalized algebraic rational identity]{Generalized algebraic rational identities\\ of subnormal subgroups in division rings} 
\author[Bui Xuan Hai]{Bui Xuan Hai}\author[Mai Hoang Bien]{Mai Hoang Bien}\author[Truong Huu Dung]{Truong Huu Dung}

\maketitle

\begin{abstract} 
In this note, we introduce a new concept of a {\it generalized algebraic rational identity} to investigate the structure of division rings. The main theorem asserts that if a non-central subnormal subgroup $N$ of the multiplicative group $D^*$ of a division ring $D$ with center $F$ satisfies a non-trivial generalized algebraic rational identity of bounded degree, then $D$ is a finite dimensional vector space over $F$. This generalizes some previous results. 
\end{abstract}
\vspace*{1cm}

Let $F$ be a field and $A$ be an $F$-algebra. If $A$ is finite dimensional over $F$, then $A$ is both algebraic and finitely generated as $F$-algebra. In 1941, Kurosh posed a famous problem (\cite[Problem R]{Kurosh}) by asking that  whether $A$ is finite dimensional provided $A$ is algebraic and finitely generated over $F$. It is well-known that the Kurosh Problem  was solved negatively by Golod and Shafarevich in \cite{gol-sha}: they have constructed  an example of an infinite dimensional  finitely generated algebraic algebra. However, the particular case when $A$ is a division ring remains unsolved up to present: we do not know whether there exists a division ring algebraic over the center and finitely generated which is infinite dimensional. This case is usually refered as the Kurosh Problem for division rings \cite[Problem K]{Kurosh}. Now, it is natural to ask what properties whose occurrence in an algebra $A$ would lead $A$ to be finite dimensional over its center. In \cite{kap}, Kaplansky proved that every primitive algebraic algebra of bounded degree is finite dimensional (see \cite[Theorem 4]{kap}). Recently, Bell et {\it al.} considered the left algebraicity and they proved the analogue theorem for division rings with this property \cite{bell}. In \cite{ak}, it was proved that if in a division ring $D$ with center $F$ all additive commutators $xy-yx$ are algebraic over $F$ of bounded degree, then $D$ is finite dimensional. An analogue result for multiplicative commutators was also obtained in \cite{chebo}. Namely, $D$ is finite dimensional in case when either $xyx^{-1}y^{-1}$ are algebraic over $F$ of bounded degree for all $x,y\in D^*$ or $\Char (D)=0$ and there exists a non-central element $a$ such that $axa^{-1}x^{-1}$ are  algebraic over $F$ of bounded degree for all $x\in D$. Further, Markar-Limanov, Chiba \cite{chiba,Pa_Ma_88} and some other authors study more general problem. In fact, they study the properties whose occurrence in a subnormal subgroup of $D^*$ entails $D$ to be finite dimensional over $F$. In particular, it was proved in \cite{chiba,Pa_Ma_88} that if $F$ is infinite and $D$ contains a non-central subnormal subgroup satisfying some generalized rational identities, then $D$ is finite dimensional over $F$. 

The aim of the present note is to give a property which would include altogether the properties given for subgroups in the results mentioned  above. This idea leads us to the concept of {\it generalized algebraic rational identities} in a non-central subnormal subgroup of $D^*$. From the main result we get in Theorem \ref{thm:1.1} below, one can easily deduce several previous well-known results. 

Let $D$ be a division ring with  center $F$ and $F\langle X \rangle$ be a free $F$-algebra on a finite set $X=\{x_1,x_2,\ldots,x_m\}$ of non-commuting indeterminates $x_1,x_2,\ldots,x_m$. We denote by $D\langle X\rangle=D*_FF\langle X\rangle$ the free product of $D$ and $F\langle X\rangle$ over $F$ and by $D(X)$ the universal division ring of fractions of $D\langle X\rangle$.  We call an element $f(x_1,x_2,\ldots,x_m)\in D(X)$ a {\it generalized rational expression} (or {\it generalized rational polynomial}) over $D$. We refer to \cite[Chapter 7]{beidar} and \cite[Chapter 8]{Rowen} for more details on generalized rational polynomials. 

Let $S$ be a subset of $D$ and  $f(x_1,x_2,\ldots,x_m)$ be a generalized rational polynomial over $D$. We say that $f$ is a {\it generalized algebraic rational identity} (briefly, \textit{GARI}) of $S$ (or  $S$ \textit{satisfies the GARI} $f$) if $ f(c_1,c_2,\cdots,c_m)$ is algebraic over $F$ whenever $f$ is defined at $(c_1,c_2,\cdots,c_m) \in S^m$. A GARI $f$ of $S$ is called \textit{non-trivial} if there exists a division ring $D_1$ with center $F$ such that $D_1$ contains all coefficients of $f$ and $f$ is not a GARI of $D_1$. We say that $f(x_1,x_2,\cdots,x_m)$ is  a \textit{generalized rational identity} (briefly, \textit{GRI}) of $S$ over $D$ (or $S$ satisfies GRI $f=0$) if $f(c_1,c_2,\cdots,c_m)=0$ whenever $f$ is defined at $(c_1,c_2,\cdots,c_m)\in S^m$. Clearly, if $f$ is a GRI of $S$, then $f$ is a GARI of $S$. A GRI $f$ of $S$ over $D$ is called \textit{non-trivial} if there exists a division ring $D_1$ containing all coefficients of $f$ such that $f$ is not a GRI of $D_1$. Observe that in a division ring $D$ which is infinite dimensional over its center, if $S$ is a non-central subnormal subgroup of the multiplicative group $D^*$ of $D$, then there exists $(c_1,c_2,\cdots,c_m)\in S^m$ such that $f$ is defined at $(c_1,c_2,\cdots,c_m)$ \cite{chiba}. 

For a given positive integer $n$, let  $x,y_1,\ldots, y_n$ be  $n+1$ non-commuting indeterminates.  Consider the following generalized rational expression 
$$g_n(x,y_1,y_2,\ldots, y_n)=\sum\limits_{\sigma  \in {S_{n + 1}}} {\sign(\sigma ){x^{\sigma (0)}}{y_1}{x^{\sigma (1)}}{y_2}{x^{\sigma (2)}} \ldots {y_n}{x^{\sigma (n)}}}, $$ where $S_{n+1}$ is the symmetric group defined on the set  $\{\,0,1,\ldots, n\,\}$.  The following result is standard and describes algebraic elements of bounded degree.

\begin{lemma}\label{2.3} Let $D$ be a division ring with center $F$. For any element $a\in D$, the following statements are equivalent:
	\begin{enumerate}
		\item The element $a$ is algebraic over $F$ of degree $\le n$.
		\item $g_n(a,y_1,y_2,\ldots, y_n)=0$ is a GRI on $D$.
	\end{enumerate} 
\end{lemma}
\begin{proof} See \cite[Corollary 2.3.8]{beidar}.
\end{proof}

Let $D$ be a division ring with center $F$ and $\phi$ be a ring automorphism of $D$. We write $D((\lambda,\phi))$ for the ring of skew Laurent series $\sum\limits_{i = n}^\infty  {{a_i}{\lambda^i}}$, where $n\in \Z, a_i\in D$, with the multiplication defined by the twist equation $\lambda a=\phi(a)\lambda$ for every $a\in D$. If $\phi=Id_D$, then we write $D((\lambda))$ instead of $D((\lambda,Id_D))$. It is known that  $D((\lambda,\phi))$ is a division ring (see \cite[Example 1.8]{lam}). Moreover, we have the following results.

\begin{lemma}\label{lem:2.1}{\rm \cite[Lemma 2.1]{bien_13}} Let $D$ be a division ring with center $F$. Assume that $K=\{\, a\in D\mid \phi(a)=a\}$ is the fixed division subring of $\phi$ in $D$. If the center $k=Z(K)$ of $K$ is contained in $F$, then the center of $D((\lambda,\phi))$ is 
	
	$$Z(D((\lambda,\phi)))=\left\{ {\begin{array}{*{20}{c}}
		k&{\text{ if } \phi \text{ has infinite order, }}\\
		{k(({\lambda^s}))}&{\text{ if } \phi \text{ has a finite order } s.}
		\end{array}} \right.$$
	In particular, the center of $D((\lambda))$ is $F((\lambda))$.  
\end{lemma}

\begin{lemma}\label {3.1} Let $D$ be a division ring with center $F$. An element $\alpha=a_1\lambda+a_2\lambda^2+\ldots$ in $D((\lambda))$ is algebraic over $F$ if and only if $\alpha=0$.
\end{lemma}
\begin{proof} Suppose that $\alpha \ne 0$ and  $g(x) =t_0 +t_1x + \cdots + t_n x^n \in F[x]$ is the minimal polynomial of $\alpha$ over $F$.  Then, the equality $$0=t_0 + t_1(a_1\lambda +a_2 \lambda^2 + \cdots) +t_2 (a_1\lambda +a_2 \lambda^2 + \cdots)^2  + \cdots + t_n (a_1\lambda +a_2 \lambda^2 + \cdots)^n$$ implies $t_0=0$, that is impossible. 
\end{proof}

For a division ring $D$ with the center $F$, let us consider a countable set of indeterminates $\{\,\lambda _i\mid i\in \Z\,\}$ and  a family of division rings which is constructed by setting 
$$D_0=D((\lambda _0)), D_1 =D_0((\lambda _1)),$$  $$D_{-1}=D_1((\lambda _{-1})), D_2=D_{-1}((\lambda _{2})),$$ 
for any $n>1,$ $$ D_{-n}=D_n((\lambda _{-n})),D_{n+1}=D_{-n}((\lambda _{n+1})).$$ 
Clearly, $D_{\infty}=\bigcup\limits_{n=-\infty}^{+\infty} {{D_n}}$ is a division ring. By Lemma~\ref{lem:2.1}, it is not hard to prove by induction on $n\ge 0$ that the center of $D_0$ is $F_0=F((\lambda _0))$, the center of $D_{n+1}$ is $F_{n+1}=F_{-n}((\lambda _{n+1}))$ and the center of $D_{-n}$ is $F_{-(n+1)}=F_{n+1}((\lambda _{-(n+1)}))$.  In particular, $F$  is contained in $Z(D_\infty)$. Consider the map $f: D_\infty\longrightarrow D_\infty$ which is defined by $f(a)=a$ for any $a\in D$ and $f(\lambda _i)=\lambda _{i+1}$ for any $i\in \Z$ is an automorphism of $D_\infty$. 

\begin{proposition}\label{pro:2.5} The center of $D_\infty((\lambda ,f))$ is $F$.  
\end{proposition}
\begin{proof} We note that  $D$ is the fixed division ring of $f$ in $D_\infty$. Since  $F$  is contained in the center of $D_\infty$, the automorphism $f$ has infinite order. By Lemma~\ref{lem:2.1}, $Z(D_\infty((\lambda ,f)))=F.$\end{proof}

\begin{theorem}\label{thm:3.2} Let $D$ be a division ring with center $F$ and $S$ be a subset of $D$. Assume that $f(x_1,x_2, \ldots, x_m)\in D(x_1,x_2,\cdots, x_m)\backslash D$ is a GARI of $S$. If $f(c_1,c_2,\cdots, c_m)\in F$ for some $(c_1,c_2,\cdots,c_m)\in S^m$, then $f$ is a non-trivial GARI.
\end{theorem}

\begin{proof} It suffices to find a division ring $L$ containing $D$ such that $f$ is not a GARI of $L$. Let $K=D(y_1,y_2,\cdots,y_m)$ and  $L=K_{\infty}((\lambda,f))$.  By Proposition~\ref{pro:2.5}, $Z(L)=Z(K)=Z(D)=F$. Consider the division subring $K((\lambda))$ of $L$. By  Lemma~\ref{lem:2.1}, $F((\lambda))$ is the center of $K((\lambda))$. In view of \cite[Lemma 7]{chiba_88}, 
$$f(1+y_1\lambda,1+y_2\lambda, \cdots, 1+y_m\lambda)=f(c_1,c_2,\cdots,c_m)+\sum_{j=1}^\infty f_j(y_1,y_2,\cdots, y_m)\lambda^j,$$
 where $f_j(y_1,y_2,\cdots, y_m)$ are generalized polynomials over $D$ and there is $j_0$  such that $f_{j_0}(y_1,y_2,\cdots, y_m) \not\equiv 0$. Since $f(c_1,c_2,\cdots,c_m)\in F$, if 
	$$f(1+y_1\lambda,1+y_2\lambda, \cdots, 1+y_m\lambda)$$
	is algebraic over $F$, then $\sum_{j=1}^\infty f_j(y_1,y_2,\cdots, y_m)\lambda^j$ is algebraic over $F$ too. By Lemma~\ref{3.1}, $f_j(y_1,y_2,\cdots, y_m)\equiv 0$ for every $j\ge 1$. In particular, we have 
	$$f_{j_0}(y_1,y_2,\cdots, y_m)\equiv 0,$$
 a contradiction. Thus, $f(1+y_1\lambda,1+y_2\lambda, \cdots, 1+y_m\lambda)$ is not algebraic over $F$. Therefore, $f$ is not a GARI of $L$.
\end{proof}
Recall that for a division ring $D$ with center $F$, an element $f\in D(X)$ is called a {\it generalized power central rational identity} (shortly, GPCRI) of a subset $S$ of $D$ if $f$ satisfies the following condition: if $f$ is defined at $(c_1,c_2,\cdots,c_m)\in S^m$, then $f(c_1,c_2,\cdots,c_m)^p\in F$ for some positive integer $p$ (see \cite{chiba_88}). Moreover, if $f^p\notin F$ for any positive integer $p$, then we say that $S$ satisfies a non-trivial  GPCRI $f$.
\begin{corollary}\label{cor:3.4} Let $D$ be a division ring with center $F$ and  $f(x_1,x_2,\cdots,x_m)$ be a generalized rational polynomial over $D$. Assume that $S$ is a subset of $D$ such that $f$ is defined at least at an $m$-tuple $(c_1,c_2,\cdots,c_m)\in S^m$. If $f$ is a non-trivial GPCRI of $S$, then $f$ is a non-trivial GARI of $S$.
\end{corollary}
\begin{proof} Assume that $f$ is a non-trivial GPCRI of $S$. Then, it is obviously that $f$ is a GARI of $S$. Now we will show that $f$ is a non-trivial GARI of $S$. It suffices to prove that $g=f^p$ is a non-trivial GARI of $S$ for some positive integer $p$. Assume that $(c_1,c_2\cdots,c_m)\in S^m$ such that $f$ is defined at $(c_1,c_2,\cdots, c_m)$. Then there exists $p>0$ such that $f(c_1,c_2,\cdots,c_m)^p\in F$. Put $g=f^p$. It is clear that $g$ is a GARI of $S$ and $g(c_1,c_2,\cdots,c_m)\in F$. One has $g\notin D$ (since $f$ is non-trivial GPCRI of $S$), so that $g$ is a non-trivial GARI of $S$.
\end{proof}

 Let $D$ be a division ring with  center $F$ and $S$ be a  subset of $D^*$. We say that $S$ satisfies a non-trivial GARI of bounded degree if there exists  a non-trivial GARI $f(x_1,x_2,\ldots,x_m)$ of $S$ over $D$ such that for all $(c_1, \ldots, c_m)\in S^m, f(c_1, \ldots, c_m)$ are algebraic over $F$ of bounded degree whenever $f(c_1, \ldots, c_m)$ are defined.

\begin{corollary}\label{cor:3.5} Let $D$ be a division ring with center $F$ and $f(x_1,x_2,\cdots,x_m)$ be a generalized rational polynomial over $D$. Assume that $S$ is a subset of $D$ such that $f$ is defined at least at an $m$-tuple $(c_1,c_2,\cdots,c_m)\in S^m$. If $f$ is a non-trivial GRI of $S$, then $f$ is a non-trivial GARI of $S$ of degree $1$.
\end{corollary}
\begin{proof} Assume that $f$ is a non-trivial GRI of $S$. Then, $f(c_1, \ldots, c_m)=0$, where $(c_1, \ldots, c_m)\in S^m$ such that $f$ is defined at $(c_1,c_2,\cdots,c_m)$. If $f\in D$, then $f=0$ that is impossible since $f$ is non-trivial GRI of $S$. Hence $f\not\in D$. In view of Theorem \ref{thm:3.2}, it follows that $f$ is a non-trivial GARI of $S$ of degree $1$.
\end{proof}

Now, we are ready to get the important theorem in this note.

\begin{theorem}\label{thm:1.1} Let $D$ be a division ring with infinite center $F$ and assume that $N$ is a non-central subnormal subgroup of $D^*$. 
	If $N$ satisfies a non-trivial GARI of bounded degree $d$, then $D$ is centrally finite, i.e. $D$ is a finite-dimensional vector space over $F$.
\end{theorem}
\begin{proof} Assume that $N$ satisfies the non-trivial GARI $f(x_1,x_2,\cdots, x_m)$ of degree $\le d$. Consider 
$$g_d(x,y_1,y_2,\cdots, y_d)=\sum\limits_{\sigma  \in {S_{d + 1}}} {\sign(\sigma ).{x^{\sigma (0)}}{y_1}{x^{\sigma (1)}}{y_2}{x^{\sigma (2)}} \ldots {y_d}{x^{\sigma (d)}}}$$ as in Lemma \ref{2.3},  and put $$w(x_1,x_2,\cdots,x_m, y_1,y_2,\cdots, y_d)=g_d(f(x_1,x_2,\cdots,x_m),y_1,y_2,\cdots, y_d).$$ 
Assume that $(c_i)\in N^m$ such that $f$ is defined at $(c_i)$. Since $f(c_1,c_2,\cdots,c_m)$ is algebraic over $F$ of degree $\le d$, by Lemma~\ref{2.3}, $$w(c_1,c_2,\cdots,c_m, r_1,r_2,\cdots, r_d)=0$$ for every $(r_i)\in D^d$. In particular, $w=0$ is a GRI of $N$. Since $f$ is a non-trivial GARI, there exists a division ring $D_1$ with center $F$ such that $D_1$ contains all coefficients of $f$ and $f$ is not a GARI of $D_1$. By Lemma~\ref{2.3}, $w\not\equiv 0$, which implies that $w=0$ is a non-trivial GRI of $N$. Now in view of \cite[Theorem 1]{chiba}, $D$ is centrally finite. Thus, the proof of Theorem \ref{thm:1.1} is now complete.
\end{proof}

In view of  corollaries~\ref{cor:3.4} and~\ref{cor:3.5}, one can see that the result of Theorem~\ref{thm:1.1} is more general than the results obtained in \cite[Theorem 1]{chiba} and \cite{Pa_Ma_88}. 

\begin{remark} 1) The  non-triviality of $f$  in Theorem~\ref{thm:1.1} is essential. For instance, let $D$ be a centrally infinite division ring,  and assume that $a\in D^*$ is an algebraic element of degree $\le d$ over the center $F$ of $D$. Then, for any $b\in D^*$, the element $bab^{-1}$ is always algebraic over $F$ of degree $\le d$. This means that $f(x)=xax^{-1}$ is a GARI of $D^*$ of degree $d$ while $[D: F]=\infty$.  
	\bigskip
		
		2) Thereom~\ref{thm:1.1} does not give any evaluation of $\dim_FD$ in terms of $d$ yet. This problem seems to be quite interesting and remains still open in general even in the case when the GARI is a GPCRI \cite[Page 137]{her2}. However, this estimation was obtained for some particular cases of GARIs of bounded degree over $D^*$. For instance, in \cite[Theorem 4]{ak}, it was proved that if  all additive commutators $xy-yx$ ($x,y\in D^*$) are algebraic over $F$ of bounded degree $d$, then $\dim_FD\le \frac{(1+d)^4}{4}$. Later, in \cite[Theorem 3]{chebo}, a better  evaluation $\dim_FD\le d^2$ was obtained. Also in \cite{chebo}, there are analogue results for multiplicative commutators in a division ring. Namely, if there exists an integer $d$ such that $\dim_FF(xyx^{-1}y^{-1})\le d$ for all $x,y\in D^*$, then $\dim_FD\le d^2$ \cite[Theorem 6]{chebo}.  In the case when $\Char(D)=0$, the result is even stronger, stating that if there exists a non-central element $a\in D$ and an integer $d$ such that $\dim_FF(axa^{-1}x^{-1})\le d$, then $\dim_FD\le d^2$ \cite[Theorem 7]{chebo}. Skipping the evaluation of $\dim_FD$ in terms of $d$, we shall extend partially the above results for a non-central subnormal subgroup $N$ of $D^*$.		
\end{remark} 

The following corollary is a broad extension of \cite[Corollary 3]{mah} and the Jacobson Theorem \cite[Theorem 7]{jab}.
\begin{corollary}\label{cor:1.1} Let $D$ be a division ring with center $F$ and assume that $N$ is a non-central subnormal subgroup of $D^*$. If all elements of $N$ are algebraic over $F$ of bounded degree $d$, then $D$ is centrally finite. Moreover, under the additional condition $\Char(D)=0$ and $N$ is normal in $D^*$, we have $\dim_FD\le d^2$.
\end{corollary}
\begin{proof}If the center $F$ is finite then every element of $N$ is torsion. By \cite[Theorem 8]{her}, $N$ is central which contradicts the hypothesis. Hence, $F$ is infinite. By Theorem~\ref{thm:1.1}, $D$ is centrally finite.
	
	Assume that $\Char(D)=0$ and $N$ is normal in $D^*$. Now let $a\in N\backslash F$. For any $x\in D^*$, $axa^{-1}x^{-1}\in N$ is algebraic over $F$ of degree $\le d$. By \cite[Theorem 7]{chebo}, $\dim_FD\le d^2$. The corollary is completed. \end{proof}

The following result extends partially \cite[Theorem 4]{ak}.
\begin{corollary}\label{cor:1.3} Let $D$ be a division ring with infinite center $F$ and assume that $N$ is a non-central subnormal subgroup of $D^*$. If  $xy-yx$ are algebraic over $F$ of bounded degree for all $x, y\in N$, then $D$ is centrally finite.
\end{corollary}
\begin{proof} The corollary is from directly Theorem~\ref{thm:1.1}. \end{proof}

The next two corollaries are partially generalizations of \cite[Theorems 6 and 7]{chebo}.

\begin{corollary}\label{cor:1.2} Let $D$ be a division ring with center $F$ and assume that $N$ is a non-central subnormal subgroup of $D^*$. If $\Char(D)=0$ and there exists an element $a\not\in F$ such that $axa^{-1}x^{-1}$ are algebraic over $F$ of bounded degree for all $x\in N$, then $D$ is centrally finite.
\end{corollary}
\begin{proof} Put $w(x)=axa^{-1}x^{-1}$. Then $w$ is a GARI of $N$ of bounded degree. Using Theorem~\ref{thm:3.2}, $w$ is a non-trivial GARI of $N$ because $w(1)=1\in F$. Since $F$ is infinite, by Theorem~\ref{thm:1.1}, $D$ is centrally finite.
\end{proof}

\begin{corollary}\label{cor:1.2} Let $D$ be a division ring with center $F$ and assume that $N$ is a non-central subnormal subgroup of $D^*$. If $xyx^{-1}y^{-1}$ is algebraic over $F$ of bounded degree for any $x,y\in N$, then $D$ is centrally finite.
\end{corollary}
\begin{proof} Put $w(x,y)=xyx^{-1}y^{-1}$. Then $w$ is a GARI of $N$ of bounded degree. Using Theorem~\ref{thm:3.2}, $w$ is a non-trivial GARI of $N$ because $w(1,1)=1\in F$. By Theorem~\ref{thm:1.1}, it suffices to show $F$ is infinite. Indeed, assume that $F$ is finite. Then, for any $a,b\in N$, the subfield $F(aba^{-1}b^{-1})$ of $D$ generated by $aba^{-1}b^{-1}$ over $F$ is finite, which implies that $aba^{-1}b^{-1}$ is torsion of order $n\le |F|^d-1$. Here, $|F|$ is the cardinality of $F$. Therefore, $N$ satisfies a generalized group identity $w(x,y)^n=1$. If there exists $a,b\in N$ such that $aba^{-1}b^{-1}=w(a,b)\not\in F$, then there exists a division subring $D_1$ of $D$ with center $F_1$ satisfying $D_1$ is centrally finite and $aba^{-1}b^{-1}\not \in F_1$ (\cite[Proposition 2.1]{bd}). Because $N$ is subnormal in $D^*$, there exist subgroups $N_1,N_2,\cdots,N_r$ of $D^*$ such that $$N=N_r \triangleleft N_{r-1}\triangleleft\cdots \triangleleft N_1=D^*.$$ Put $H_i=D_1\cap N_i$. Then we obtain that $$H_r \triangleleft H_{r-1}\triangleleft\cdots \triangleleft H_1=D_1^*.$$ It implies that $H_r$ is a subnormal subgroup of $D_1^*$ and $aba^{-1}b^{-1}\in H_r$. Observe that $F_1$ is infinite and since $H_r\subseteq N$ satisfies the group identities $w(x,y)^n=1$, by \cite[Theorem 3.1]{Pa_Bi_15}, $H_r\subseteq F_1$. In particular, $aba^{-1}b^{-1}\in F_1$. Contradiction! Hence, $aba^{-1}b^{-1}\in F$ for any $a,b\in N$, which implies that $N$ is solvable. In the view of \cite[14.4.4]{scott}, $N$ is central, which contradicts to the hypothesis. Thus, $F$ is infinite and the corollary is proved completely.
\end{proof}

\vspace*{0.5cm}

Bui Xuan Hai,                                                          \hspace*{4.2cm}Mai Hoang Bien,

Faculty of Mathematics and                                     \hspace*{2.0cm} Department of Basic Sciences 

Computer Science,                                                  \hspace*{3.4cm} University of Architecture,  

University of Science, VNU-HCM                         \hspace*{1.1cm} 196 Pasteur Str., Dist. 1, HCM-City, 

227 Nguyen Van Cu Str., Dist. 5,                           \hspace*{1.1cm} Vietnam

HCM-City, Vietnam                                               \hspace*{3.3cm}e-mail: maihoangbien012@yahoo.com

e-mail: bxhai@hcmus.edu.vn         \\ \\

Truong Huu Dung

Faculty of Mathematics and Computer Science

University of Science, VNU-HCM

227 Nguyen Van Cu Str., Dist. 5,                          

HCM-City, Vietnam 

e-mail: thdung@dnpu.edu.vn 
\end{document}